\documentclass{article}
\usepackage{graphicx}
\usepackage{amsthm, amsmath, amssymb, tikz, bm}
\usepackage{mathtools}
\usepackage{xifthen}
\usepackage[normalem]{ulem}\usepackage{bbm}
\usepackage{float}
\usepackage{subcaption}

\usepackage[margin=1.2in]{geometry} 

\usepackage[shortlabels]{enumitem}
\usepackage{todonotes}
\usepackage[colorlinks=true,
linkcolor=blue,citecolor=blue,
urlcolor=blue]{hyperref}

\allowdisplaybreaks

\usetikzlibrary{calc,shapes, backgrounds}

\usepackage{amssymb, amsmath, amsthm, graphicx,mathrsfs}
\usepackage{xcolor}

\usepackage{graphicx} 

\newtheorem{thm}{Theorem}[section]
\newtheorem{cor}[thm]{Corollary}

\newtheorem{lem}[thm]{Lemma}
\newtheorem{conj}[thm]{Conjecture}
\newtheorem{construction}[thm]{Construction}

\newtheorem{claim}[thm]{Claim}

\title{The Tur\'an number of Berge paths in connected hypergraphs}
\date{\today}
\author{
{{Yupei Li}}\thanks{
\footnotesize {University of South Carolina, Columbia, SC 29208, USA. Email: {\tt yupei@email.sc.edu}.}}
\and{{Ruth Luo}}\thanks{University of South Carolina, Columbia, SC 29208, USA. Email: {\tt ruthluo@sc.edu}. Research of this author
is supported in part by NSF grant DMS-2452134.
}}
\begin{document}

\maketitle

\begin{abstract} 
A Berge path in a hypergraph is an alternating sequence of distinct vertices and edges $v_1,e_1, v_2, \ldots, e_{k-1}, v_{k}$ such that $v_i, v_{i+1} \in e_i$ for all $i \in [k-1]$. Gy\H{o}ri, Katona, and Lemons determined the maximum number of edges in $r$-uniform, $n$-vertex hypergraphs without Berge paths on $k$ vertices. The extremal examples for their bounds are disconnected hypergraphs in which each component has either too few edges or too few vertices to contain such a Berge path.

We consider the setting in which the host hypergraph is connected. Let $\operatorname{ex}^{\mathrm{conn}}_r(n,\mathcal{BP}_k)$ denote the maximum number of edges in an $n$-vertex, $r$-uniform, connected hypergraph that does not contain a Berge path on $k$ vertices. This number was determined by Zhang, Broersma, Győri, Tompkins, Wang for $k \leq r+2$, and Zhao, Gerbner, Zhou for $k \geq 2r+2$. In this paper, we resolved the remaining cases $r+3 \leq k \leq 2r+1$ when $n$ is sufficiently large, confirming a conjecture of  Zhang, et al.
\end{abstract}

\section{Introduction}
For a fixed graph $H$, the {\em Tur\'an number} of $H$ is the function $\operatorname{ex}(n,H)$ equal to the maximum number of edges in an $n$ vertex, $H$-free graph. In 1959, Erd\H{o}s and Gallai~\cite{EG} determined the following bound for the Tur\'an number of paths and long cycles. Let $P_k$ denote the path on $k$ vertices (i.e., of length $k-1$) and let $C_k$ denote the cycle on $k$ vertices. The {\em circumference of a graph} is the length of its longest cycle.

\begin{thm}[Erd\H{o}s--Gallai \cite{EG}] Fix $n \geq k \geq 2$. 
    Then $\operatorname{ex}(n,P_k) \leq \frac{(k-2)n}{2}$. Moreover, if $n \geq k \geq 3$, and $G$ has circumference less than $k$, then $|E(G)| \leq \frac{(n-1)(k-2)}{2}$. 
\end{thm}

The bound for paths is best possible whenever $n$ is divisible by $k-1$: take the disjoint union of cliques $K_{k-1}$. For the circumference result, when $n-1$ is divisible by $k-2$ we take any graph in which every block is a clique $K_{k-1}$. 

We note that the paths result is an easy corollary of the circumference result: let $G$ be a graph without a copy of $P_k$, and let $G'$ be obtained from $G$ by creating a new vertex $v$ that is adjacent to all of $V(G)$. A cycle in $G'$ of length at least $\ell$ yields a copy of $P_{\ell-1}$ in $G$, so $G'$ has circumference at most $k$. 

Various strengthenings of the Erd\H{o}s--Gallai problem were proven by Faudree and Schelp~\cite{FS, FS2}, Lewin~\cite{Lewin}, Woodall~\cite{Woodall}. The strongest such result was due to Kopylov~\cite{K}.

Observe that the extremal example for the path Tur\'an problem is disconnected. By $\operatorname{ex}^{\mathrm{conn}}(n, P_k)$, we denote the maximum number of edges in a {\em connected} $n$-vertex $P_k$-free graph. 

\begin{thm}[Kopylov~\cite{K}]\label{Kopylovthm} Let $n \geq k \geq 4$. Then \[\operatorname{ex}^{\mathrm{conn}}(n, P_k) \leq \max\{{k-2 \choose 2} + (n-k+2), {\lfloor k/2 \rfloor \choose 2} + \lfloor (k-2)/2 \rfloor (n-\lfloor k/2\rfloor)\}.\] 
\end{thm}Kopylov also proved exact values for the maximum number of edges in a $2$-connected, $n$-vertex graph with circumference  less than $k$. 

We consider hypergraph versions of these problems.

\subsection{Berge paths in hypergraphs}
 Unless stated otherwise, our hypergraphs are {\em simple}, that is, they do not contain multi-edges.

A {\em hypergraph $H$} is set of vertices $V(H)$ and a set of edges $E(H)$ in which each edge is a subset of vertices. When each edge has the same size, say $r$, we say $H$ is {\em $r$-uniform}.

A {\em Berge path} on $k$ vertices is an alternating sequence of $k$ distinct vertices and $k-1$ distinct edges 
\[P=v_1, e_1, v_2, \ldots, e_{k-1}, v_k,\] in which $v_{i}, v_{i+1} \in e_i$ for all $1 \leq i \leq k-1$. 

Although Berge paths may contain many vertices in the union of its edges, we simply define $V(P) = \{v_1, \ldots, v_k\}$ and $E(P) = \{e_1, \ldots, e_{k-1}\}$.

Similarly, a {\em Berge cycle} of length $k$ is an alternating sequence of $k$ distinct vertices and $k$ distinct edges
\[C=v_1, e_1, v_2, \ldots, e_{k-1}, v_k, e_k, v_1,\] in which $v_i,v_{i+1} \in e_i$ for all $1 \leq i \leq k$ (with indices taken modulo $k$). 

We define $V(C) = \{v_1, \ldots, v_k\}, E(C) = \{e_1, \ldots, e_k\}$.

We say a hypergraph is {\em connected} if every pair of its vertices can be connected by a Berge path.

The first analogue of the Erd\H{o}s--Gallai Theorem for hypergraphs was proven by Gy\H{o}ri, Katona, and Lemons~\cite{GKL}. We note that in this result and the results to follow, the bound on the Tur\'an numbers depend on the relationship between the uniformity $r$ and the forbidden path size $k$.

Let $\mathcal{BP}_k$ be the set of Berge paths on $k$ vertices. We use $\operatorname{ex}_r(n, \mathcal{BP}_k) \;\;(\operatorname{ex}^{\mathrm{conn}}_r(n, \mathcal{BP}_k))$ to denote the maximum number of edges in an $n$-vertex, $r$-uniform (connected) hypergraph with no $\mathcal{BP}_k$. 

\begin{thm}[Gy\H{o}ri, Katona, Lemons~\cite{GKL}] If $n\geq k > r+2\geq 5$, then $\operatorname{ex}_r(n, \mathcal{BP}_k) \leq \frac{n}{k-1}{k-1 \choose r}$, and if $n\geq r+1 \geq k \geq 3$, then $\operatorname{ex}_r(n, \mathcal{BP}_k) \leq \frac{n(k-2)}{r+1}$.
\end{thm}
Later,  Davoodi, Gy\H{o}ri, Methuku, and Tompkins~\cite{DGMT} completed the missing case $k = r+2$.

Similar to the graph case, for $k \geq r+2$, the extremal example is the hypergraph composed of disjoint cliques $K_{k-1}^{(r)}$. The behavior changes for hypergraphs with large uniformity. For $r + 1 \geq k$, we instead take the hypergraph in which every connected component has $r+1$ vertices and $k-2$ arbitrary edges. Since a Berge path on $k$ vertices has $k-1$ edges, such a hypergraph is $\mathcal{BP}_k$-free.

In this paper, we consider the connected version of the Tur\'an problem. The following are the previously known results for $\operatorname{ex}^{\mathrm{conn}}_r(n,\mathcal{BP}_k)$.

\begin{thm}[Zhang, Broersma, Győri, Tompkins, Wang \cite{ZBGTW}]
Fix integers $n,r,k$ with $r+1\ge k\ge 5$. When $n$ is sufficiently large and $n$ is not a multiple of $r$, we have
\[
\operatorname{ex}^{\mathrm{conn}}_r(n,\mathcal{BP}_k)
=
\left\lfloor \frac{k-2}{2} \right\rfloor
\left\lfloor \frac{n-1}{r} \right\rfloor
+\mathbf{1}_{2\nmid k},
\]
where $\mathbf{1}_{2\nmid k}=1$ if $2\nmid k$ and
$\mathbf{1}_{2\nmid k}=0$ otherwise.
\end{thm}

This Theorem is sharp for infinitely many $n$ due to the following construction $H$. Let $n = tr + 1$ for some $t \in \mathbb N$. The vertex set $V(H)$ consists of $t$ sets $V_1, V_2, \ldots, V_t$ each of size $r+1$, and there exists a vertex $v \in V(H)$ with $V_i \cap V_j=\{v\}$ for all $i\neq j$. Inside each $V_i$ there are $\lfloor (k-2)/2 \rfloor$ edges, and if $k$ is odd, then (only) $V_1$ contains an extra edge. 

A Berge path in $H$ can only intersect at most $2$ $V_i$'s, and hence contains at most $2\lfloor (k-2)/2 \rfloor + \mathbf{1}_{2\nmid k}\leq k-2$ edges.

\begin{thm}[Zhang, Broersma, Győri, Tompkins, Wang \cite{ZBGTW}]\label{sunflowerthm}
Let $r\geq 3$. When $n$ is sufficiently large, we have
\[
\operatorname{ex}^{\mathrm{conn}}_r(n,\mathcal{BP}_{r+2})
= n-r+1.
\]
\end{thm}

The {\em sunflower} $\mathcal S_{n,r}$ with core of size $r-1$ (that is, $n-r+1$ edges which share a common $r-1$ vertices and each contain a degree 1 vertex) shows that this Theorem is best possible: in $\mathcal S_{n,r}$, only $r-1$ vertices have degree at least $2$, hence a longest Berge path has $r-1 + 2 =r+1$ vertices.

The case for $k$ large compared to $r$ was first proved by F\"uredi, Kostochka, and the second author~\cite{FKL} for $k \geq 4r$, then Gy\H{o}ri, Salia, and Zamora~\cite{GSZ} for $k \geq 2r+13$. Very recently, this was extended to $k \geq 2r+2$ by Zhao, Gerbner, and Zhou~\cite{ZGZ}. 
\begin{thm}[Zhao, Gerbner, Zhou \cite{ZGZ}]
For all integers $n$, $k$ and $r$, there exists $N_{r,k}$ such that if
$n>N_{r,k}$ and $k\geq 2r+2\geq 8$, then
\[
\operatorname{ex}^{\mathrm{conn}}_r(n,\mathcal{BP}_k)
=
\binom{\left\lfloor k/2\right\rfloor-1}{r-1}
\left(n-\left\lceil k/2\right\rceil\right)
+
\binom{\left\lceil k/2\right\rceil}{r}.
\]
\end{thm}

Their bound is sharp due to the following construction $H(n,k,r)$. The vertex set is partitioned into sets $L$ and $R$ with $|L|=\lfloor k/2\rfloor-1$. The edges consist of all $r$-sets with at least $r-1$
vertices in $L$ and at most one vertex from $R$. When
$k$ is odd, we also add the $r$-sets containing two fixed vertices
$u_1,u_2 \in R$ and $r-2$ vertices in $L$.

A Berge path in $H(n,k,r)$ contains at most $\lfloor k/2 \rfloor-1 +1+ \mathbf{1}_{2\nmid k}$ vertices from $R$ and hence contains at most $2(\lfloor k/2 \rfloor-1)+1 + \mathbf{1}_{2\nmid k} =k-1$ vertices.

\medskip

See~\cite{Ergemlidze, FKL, FKL2, GLSZ, KL} for related results on long Berge cylces in hypergraphs. 


\section{Main Theorem and proof discussion}

In~\cite{ZBGTW}, Zhang et al. conjectured that $\operatorname{ex}^{\mathrm{conn}}_r(n,\mathcal{BP}_k) =n-(k-3) + {k-3 \choose 2}$ for the cases $r+3\leq k \leq 2r$ when $n$ is sufficiently large compared to $r$ and $k$. Moreover, Zhao et al.~\cite{ZGZ} asked also about the case $k=2r+1$. 

In this paper, we confirm the conjecture of Zhang et al., providing the exact value for the remaining cases $r+3 \leq k \leq 2r+1$ and completing the bounds for $\operatorname{ex}^{\mathrm{conn}}_r(n, \mathcal{BP}_k)$ when $n$ is large.  
\begin{thm}\label{thm:main}
 Fix integers $k$ and $r$ such that $r+3 \leq k \leq 2r+1$. Then for $n \geq n_0$ with $n_0= \max\{(r-1)^2, k-3 + 2{k-3 \choose r-1}\}$,

 \[
\operatorname{ex}^{\mathrm{conn}}_{r}(n,\mathcal{BP}_k)
=
n-(k-3)+\binom{k-3}{r}.
\]
\end{thm}

This bound is achieved by the following family of hypergraphs.

\begin{construction}\label{mainextremal}The family $\mathcal H_1$ contains all $n$-vertex hypergraphs $H$ with an $r$-uniform clique $K_{k-3}^{(r)}$ on vertex set $X$ with size $k-3$, and each of the $n-(k-3)$ vertices $v$ outside of $X$ belongs to exactly one edge  which contains $v$ and some $r-1$ vertices in $X$.
\end{construction}

As a Berge path contains at most $2$ vertices with degree $1$, the longest Berge path in $H \in \mathcal H_1$ has $k-3+2 =k-1$ vertices.

We remark that the sunflower $\mathcal S_{n,r}$ belongs to $\mathcal H$ when $k=r+3$ (in this case, $n-(k-3) + {k-3 \choose r} = n-r+1$). Thus Theorem~\ref{thm:main} can also be viewed as a slight strengthening of Theorem~\ref{sunflowerthm}: any connected $r$-uniform hypergraph with more than $|E(\mathcal S_{n,r})|$ edges contain a Berge path on $r+3$ vertices.

\begin{figure}[htbp]
    \centering
    \includegraphics[scale=1]{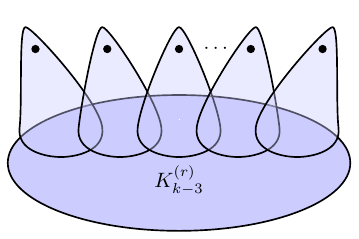} \qquad 
    \includegraphics[scale=1]{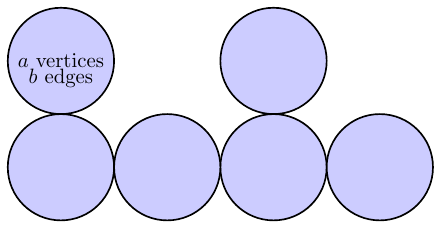}
    \caption{A hypergraph of $\mathcal H_1$ (left) and an $(a,b)$-block tree (right).}
    \label{fig:blocktree}
 \end{figure}


\subsection{Proof outline}
We begin by finding a longest Berge cycle $C$ of length at least $r$ in $H$ (which exists by some previously known results to be stated in the following subsection). If there exists two Berge cycles which are both vertex- and edge-disjoint with combined length at least $k$, we can use the connectivity of $H$ to find a Berge path on $k$ vertices.

Given $C$, we consider the subset $E'(H) \subseteq E(H)-E(C)$ of edges which contain at least one vertex outside of $V(C)$. By considering the partial edges $\{e \setminus V(C): e \in E'(H)\}$ we obtain a non-uniform multi-hypergraph on vertex set $V(H) - V(C)$ in which each edge contains at most $r$ edges. We show that in this new hypergraph, there exists a ``block" with many edges. We find this subhypergraph through a pruning process in which we repeatedly delete ``leaf blocks" containing few edges.
We find a Berge cycle in our block that combines with $C$ (via an application of Menger's Theorem) to form either a longer Berge cycle than $C$ or a Berge path on $k$ vertices.

We note that in some parts of the proof, the case $k=r+3$ behaves differently than $k \geq r+4$. In particular, in the extremal bound $n-(k-3) + {k-3 \choose r}$ we have ${k -3 \choose r} = 1$. The bound is considerably smaller than when $k \geq r+4$. As such, some parts of the proof for the case $k=r+3$ will be relegated to the Appendix, as to not distract from the main proof.

\subsection{Tools and Lemmas}
The {\em 2-shadow} of a hypergraph $H$ is the graph $\partial_2H$ on vertex set $V(H)$ in which $xy \in E(\partial_2 H)$ if and only if there is an edge $e \in E(H)$ containing $\{x,y\}$. 
For integers $a,b,r$ with $a \geq r$, an $r$-uniform hypergraph $H$ is called an {\em $(a,b)$-block tree} if $\partial_2H$ is connected and every $2$-connected block of $\partial_2H$ has exactly $a$ vertices and induces exactly $b$ edges in $H$. See Figure~\ref{fig:blocktree}.

We will use the following results in our proof to find long Berge cycles.



\begin{thm}[Győri, Lemons, Salia, Zamora \cite{GLSZ}]\label{thm:r}
Let $r\geq2$ and $n$ be positive integers.
 $H$ is an $r$-uniform, $n$-vertex hypergraph with no Berge cycle of length $r$ or longer, then
\[|E(H)| \leq \max\left\{
\left\lfloor\frac{n-1}{r}\right\rfloor(r-1),
\,
n-r+1
\right\}.\]
When $n-r+1>\left\lfloor\frac{n-1}{r}\right\rfloor(r-1)$ (roughly when $n > (r-1)^2)$ the only extremal graph is the sunflower $\mathcal S_{n,r}$. When
$\left\lfloor\frac{n-1}{r}\right\rfloor(r-1)>n-r+1$ and $r\mid(n-1)$ the only
extremal graphs are the $(r+1,r-1)$-block trees.
\end{thm}

When we allow $H$ to have multi-edges, we obtain a single class of extremal examples.
\begin{thm}[Győri, Lemons, Salia, Zamora \cite{GLSZ}]\label{thm:multi cycle}
Let $k$, $n$ and $r$ be positive integers such that $2\le k\le r \leq n$. If $H$ is an $r$-uniform multi-hypergraph with no Berge cycle of length $k$ or longer, then
\[
|E(H)| \leq 
\left\lfloor\frac{n-1}{r-1}\right\rfloor(k-1).
\]

If $r-1\mid(n-1)$ the only extremal graphs with $n$ vertices are the
$(r,k-1)$-block trees.    
\end{thm}

This implies the following easy corollary which we will use throughout our proof.

\begin{cor}\label{cor:multi cycle}
Let $n$ and $r$ be positive integers such that $n \geq r \geq 2$. If $H$ is an $n$-vertex multi-hypergraph such that each edge has cardinality at least $r$ and
\[
|E(H)|
\geq n-1\]
then either $H$ is $r$-uniform and an $(r,r-1)$-block tree on exactly $n-1$ edges, or $H$ contains a Berge cycle of length at least $r$. 
\end{cor}


\begin{proof}
Let $H$ be the multi-hypergraph such that each edge has cardinality at least $r$. Let $H'$ be the $r$-uniform multi-hypergraph obtained from $H$ by replacing every edge with an arbitrary $r$-subset of it.  Then $|E(H')|=|E(H)| \geq n-1$. Applying Theorem \ref{thm:multi cycle} to $H'$ with $k=r$, we have that $H'$ (and therefore $H$) contains a Berge cycle of at least $r$ vertices, or $H'$ is an $r$-uniform, $(r,r-1)$-block tree. In the case $H'$ is a $(r,r-1)$-block tree, if $H$ is $r$-uniform, then $H=H'$. If $H$ is non-uniform, then there exists an edge $f' \in E(H')$ such that $f' \subsetneq f$ for some $f \in E(H)$. Since $f$ has cardinality at least $r+1$, there exists a vertex $v \in f \setminus f'$ which belongs to a different block of the block tree than $f'$. 
If $(r,r-1) = (2,1)$, then it is easy to check that there is a Berge cycle of length at least $2$ containing $f$, the vertex $v$, and another edge containing $v$.

Now suppose $r \geq 3$. Let $u \in f'$ be a vertex not contained in the block containing $v$. Call this block $B$. 
Since $H'$ is a $(r,r-1)$-block tree, we can find a Berge path in $H'- \{f'\}$ (and therefore in $H -\{f\}$) from $v$ to $u$ with at least $r-1$ edges in $B$ and at least one edge outside of $B$.
Joining $v$ and $u$ by $f$, we obtain a Berge cycle with at least $r$ vertices in $H$.
\end{proof}

We will also use the following lemma which controls how hyperedges intersect the vertices of a longest Berge cycle. 
\begin{lem}\label{lem:intC}
    Let $C = v_1,v_2, \ldots, v_\ell$ be a cycle in a graph, and let $X, Y \subseteq V(C)$ such that $|X|,|Y| \geq q$ and for any distinct $v_i \in X, v_j \in Y, |i-j| \geq s+1\geq 2$. Then the following holds.
    \begin{enumerate}
        \item If $|X \cap Y| \geq q \geq 2$, then $\ell \geq q(s+1)$ with equality only if $|X| = |Y| = q$. 
        \item If $|X \cap Y| \leq q-1$, then $\ell \geq 2q + 2s + |X \cap Y|(s-1)$. 
    \end{enumerate}
\end{lem}
\begin{proof}
    Part 1 follows immediately from the condition $|i-j| \geq s+1$ for all $v_i,v_j \in X \cap Y$. Now assume $|X \cap Y |\leq q-1$. In particular, there exists vertices in $X - Y$ and $Y-X$. 

    We view $C$ to be oriented clockwise. Call a segment $\{v_{i+1}, \ldots, v_{i+s}\}$ of $s$ consecutive vertices in $C$ {\em blank} if it does not intersect $X \cup Y$. For each $v_i \in X \cap Y$, the segment $\{v_{i+1}, \ldots, v_{i+s}\}$ is blank. Moreover, there exists a vertex $v_j \in X - Y$ such that the next clockwise appearing vertex in $X\cup Y$ belongs to $Y$. Hence the segment $\{v_{j+1}, \ldots, v_{j+s}\}$ is also blank. Similarly, we may find such a blank segment for a vertex in $Y-X$. Altogether, we obtain at least $|X \cap Y| + 2$ blank segments of $s$ vertices.

    We now count the vertices in $C$. Define a mapping $f$ on $X \cup Y$ by $f(v_i) = \{v_i\}$ if $v_i \in (X-Y) \cup (Y-X)$ and $f(v_i) = \{v_i, v_{i+1}\}$ if $v_i \in X \cap Y$. The sets $\{f(v_i):v_i \in X \cup Y\}$ are pairwise disjoint, and the union of them contains $|X-Y| + |Y-X| + 2|X \cap Y| \geq 2q$ vertices. Moreover, each $f(v_i)$ with $v_i \in X \cap Y$ intersects exactly one blank segment in one vertex. Therefore

    \[\ell = |V(C)| \geq 2q + (|X \cap Y|+2) s - |X \cap Y| = 2q + 2s + |X\cap Y|(s-1).\]
\end{proof}
\section{Proof of Theorem \ref{thm:main}}
Construction~\ref{mainextremal} gives the lower bound. Now we will prove the upper bound. Suppose $H$ is a counterexample to Theorem~\ref{thm:main}. That is, $|E(H)| > {k-3 \choose r} + n-(k-3)$ but $H$ contains no Berge path on $k$ vertices. If $k \geq r+4$, we have $|E(H)| \geq n$ and thus $H$ contains a Berge cycle of length at least $r$ by Corollary~\ref{cor:multi cycle}. If $k = r+3$, we have $|E(H)| \geq n-r+1$ and thus $H$ contains a Berge cycle of length at least $r$ by Theorem~\ref{thm:r} for $n \geq n_0 \geq (r-1)^2$. 


Suppose the longest Berge cycle in $H$ has $\ell$ vertices. Among all Berge cycles of $\ell$ vertices, choose $C=v_1e_1 \ldots v_{\ell}e_{\ell}v_1$ such that it contains as many edges in the set $\{e \in E(H): e \subseteq V(C)\}$ as possible. We have $\ell \geq r$.

\begin{claim} \label{cla: edges inside C}
    $|E(C) \cup \{e \in E(H): e \subseteq V(C)\}| \leq \max\{\ell, {\ell \choose r}\}$.
\end{claim}

\begin{proof}

    If $\ell = r$, then there is at most one edge in $\{e \in E(H): e \subseteq V(C)\}$. We may assume that this edge is also used in $E(C)$, otherwise we can swap an arbitrary edge in $C$ with it to obtain a Berge cycle $C'$ such that $|\{e \in E(C'): e \subseteq V(C')\}| > |\{e \in E(C): e \subseteq V(C)\}|$, contradicting the choice of $C$.
    Therefore $|E(C) \cup \{e \in E(H): e \subseteq V(C)\}| = \ell$. 

    Now suppose $r \leq \ell-1$. 
 If every edge in $C$ is a subset of $V(C)$ then we are done. So assume some edge, $e_i$, contains a vertex outside of $C$. By the choice of $C$,  every $r$-subset in $V(C)$ containing $\{v_i, v_{i+1}\}$ either is not an edge of $H$ or is already used as another edge in $C$. There are ${\ell-2 \choose r-2}$ such subsets. 
    Then 
    \[|E(C) \cup \{e \in E(H): e \subseteq V(C)\}| \leq \ell +  {\ell \choose r} - {\ell-2 \choose r-2}.\]

    If $\ell \leq {\ell-2 \choose r-2}$, then we are done. This holds in particular if $4 \leq r \leq \ell-2$. In the remainder we assume either $r = 3$ or $r = \ell-1$.

Suppose there are at least $3$ edges in $E(C)$ that contain vertices outside of $V(C)$. In particular, we can find two such edges $e_i, e_j$ satisfying $|i-j| \geq 2$. Then every $r$-subset containing $\{v_i,v_{i+1}\}$ or $\{v_j, v_{j+1}\}$ must either be used in $E(C)$ or does not belong to $E(H)$. The number of such subsets is $2{\ell-2 \choose r-2} - {\ell-4 \choose r-4} \geq 2(
\ell-2) - (\ell-4) =\ell$ if $r \geq 4$ and $2{\ell-2 \choose r-2} = 2\ell-4 \geq \ell$ if $ r = 3$.  So
\[|E(C) \cup \{e \in E(H): e \subseteq V(C)\}| \leq \ell + {\ell \choose r} - \ell = {\ell \choose r}.\]

Finally, suppose there are at most $2$ edges in $E(C)$ containing vertices outside of $V(C)$. If at least 2 subsets of $V(C)$ of size $r$ do not belong to $E(H)$, then we obtain $|E(C) \cup \{e \in E(H): e \subseteq V(C)\}| \leq 2 + {\ell \choose r} - 2$, and we are done. 
Otherwise, by taking a permutation $w_1, \ldots, w_\ell$ of $V(C)$, we may assume that if any $r$-subset of $V(C)$ is missing from $E(H)$, then it is the set $f=\{w_1, w_3, w_4, \ldots, w_{r+1}\}$. Instead of $C$ we may take the tight cycle on $\ell$ vertices with all edges containing $r$ consecutive vertices in the permutation, that is edges of the form $\{w_i, w_{i+1}, \ldots, w_{i+r-1}\}$ (with indices mod $\ell$). Note that $f$ is not a set of this form when $\ell > r+1$, and all such sets are edges in $H$ since we assumed that $f$ was the only missing $r$-subset of $V(C)$. In this case, the Berge cycle contains more edges that are subsets of $V(C)$ than $C$ does, a contradiction. If $\ell = r+1$, then we can similarly create a Berge cycle $C'$ on $V(C)$ whose edges are all $r$ subsets of $V(C)$, or which contains exactly one edge which is not a subset of $V(C)$. By the choice of $C$, $|E(C) \cup \{e \in E(H): e \subseteq V(C)\}| \leq 1 + {\ell \choose r} - 1$.
\end{proof}

 Given any hypergraph $H$, its corresponding {\em incidence bipartite graph} is a bipartite graph $I = I(H)$ where $V(I)=V(H) \cup E(H)$. For $u \in V(H)$ and $e \in E(H)$, $ue \in E(I)$ if and only if $u \in e$. Observe that a cycle in $I$ of length $2s$ corresponds to a Berge cycle of length $s$ in $H$, and vice versa. We call a path in the incidence bipartite graph a {\em semi-path} in $H$. In a semi-path, an endpoint may either be a vertex or a hyperedge in $H$. In the case where both endpoints are vertices, the semi-path is a Berge path.

\begin{claim} \label{cla: k-3}
 $\ell \leq k-3$.   
\end{claim}

\begin{proof}
Suppose $\ell \geq k-2$. First we must have $\ell < k-1$ since otherwise we could find a Berge path of $k$ vertices by the connectivity of $H$. Thus, we have $\ell=k-2$. Define \[E'(H)=\{e \in E(H)-E(C): e -V(C) \neq \emptyset\}.\] We first show that $e \cap V(C) \ne \emptyset$ for all $e \in E'(H)$. Suppose on the contrary that $e \cap V(C)=\emptyset$ for some $e \in E'(H)$, then by the connectivity of $H$, there exists a semi-path from $V(C) \cup E(C)$ to the vertices in $e$. Let $P$ be a shortest such semi-path. If the first element of $P$ is an edge, let it be $e_i$. If the first element of $P$ is a vertex, let it be $v_i$. Let $u \in e$ be the last element of $P$. Since $r \geq 2$, there exists another vertex $u' \in e$. We observe that \[v_{i+1}, e_{i+1}, \ldots, e_{i-1}, v_i, P, u, e, u'\]
is a Berge path with at least $\ell+2=k$ vertices, a contradiction. Thus, we have $e \cap V(C) \ne \emptyset$ for all $e \in E'(H)$. We show the following. 

\begin{equation}\label{eq:noint}
\mbox{For any two edges $f_1, f_2 \in E'(H)$, we have $(f_1 \cap f_2) -V(C) = \emptyset$.}\end{equation} Suppose there exists two edges such that $(f_1 \cap f_2) -V(C) \ne \emptyset$. Let $u \in (f_1 \cap f_2) -V(C)$.

\textbf{Case 1}: $|f_1 - V(C)| \geq 2$ or $|f_2 - V(C)| \geq 2$. 

Without loss of generality assume $|f_2 - V(C)| \geq 2$. Then there exists another vertex $u' \in f_2-V(C)$. Let $v_i \in f_1 \cap V(C)$. We observe that $u', f_2, u, f_1, v_i, e_i, \ldots, e_{i-2}, v_{i-1}$ is a Berge path of $\ell+2=k$ vertices, a contradiction. 

\textbf{Case 2}: $|f_1-V(C)| =|f_2 - V(C)|=1$.

We observe that if $v_i \in f_1 \cap V(C)$, then $v_{i-1},v_{i+1} \notin f_2 \cap V(C)$. Suppose on the contrary that $v_i \in f_1 \cap V(C)$ and $v_{i+1} \in f_2 \cap V(C)$, then $v_i, f_1, u, f_2, v_{i+1}, e_{i+1}, \ldots, e_{i-1}, v_i$ is a Berge cycle with $\ell+1$ vertices, a contradiction. Similar for $v_{i-1} \in f_2 \cap V(C)$. Hence, we have $|i-j| \geq 2$ for any distinct $v_i \in f_1 \cap V(C)$ and $v_j \in f_2 \cap V(C)$. Let $X=f_1 \cap V(C)$ and $Y=f_2 \cap V(C)$. Since $H$ has no multi-edges, we have $X \neq Y$, and hence $|X \cap Y| \leq (r-1)-1$. Applying Lemma \ref{lem:intC} with $q=r-1$ and $s=1$, we obtain $\ell \geq 2(r-1)+2 \geq k-1$, a contradiction. This proves~\eqref{eq:noint}. 

Next we will prove that most edges in $E'(H)$ contain at least $2$ vertices outside of $C$.

\begin{equation}
\label{eq:1edge1}\mbox{There is at most one edge $e \in E'(H)$ such that $|e - V(C)|=1$.}
\end{equation}

Suppose there are two distinct edges $e,e'\in E'(H)$ with $|e - V(C)|=1$ and $|e' - V(C)|=1$. By~\eqref{eq:noint}, there are two distinct vertices $u_1 \in e - V(C)$ and $u_2 \in e' - V(C)$. We observe that if $v_i \in e \cap V(C)$, then $v_{i-1},v_{i+1} \notin e' \cap V(C)$. Otherwise if $v_i \in e \cap V(C)$ and $v_{i+1} \in e' \cap V(C)$, then $u_1, e, v_i, e_{i-1}, \ldots, e_{i+1}, v_{i+1}, e', u_2$ is a Berge path of $\ell+2=k$ vertices, a contradiction. Similar for $v_{i-1} \in e' \cap V(C)$.

As before, we have $|i-j| \geq 2$ for any distinct $v_i \in e \cap V(C)$ and $v_j \in e' \cap V(C)$. Let $X=e \cap V(C)$ and $Y=e' \cap V(C)$. If $|X \cap Y| \leq (r-1)-1$. Applying Lemma \ref{lem:intC} with $q=r-1$ and $s=1$, we obtain $\ell \geq 2(r-1)+2 \geq k-1$, a contradiction. Thus, $X \cap Y=X=Y$ with size $r-1$. Then applying Lemma \ref{lem:intC}, we have $k-2=\ell \geq 2(r-1)$, so $2r \leq k \leq 2r+1$. By symmetry we may assume $X = Y = \{v_2, v_4, \ldots, v_{2r-2}\}$.

For distinct $i, j$, with $\{v_i, v_j\} \subseteq X$, suppose there exists some $e'' \in E(H)-E(C)$ for which $\{v_{i+1},v_{j+1}\} \subseteq e''$. Note that $e'' \notin\{e, e'\}$ since otherwise this would violate $|i-j| \geq 2$ for $j = i+1$. Then \[u_1, e, v_i, e_{i-1}, \ldots, e_{j+1}, v_{j+1}, e'', v_{i+1}, e_{i+1}, \ldots, e_{j-1}, v_j, e', u_2\]
is a Berge path with $\ell+2=k$ vertices, a contradiction. 
Therefore no such $e''$ exists. Similarly, $\{v_{i-1},v_{j-1}\} \not\subseteq e''$ for all $e'' \in E(H) - E(C)$. 

Observe that $V(C) - X = \{v_1, v_3, \ldots, v_{2r-3}\}$ if $k = 2r$ and $V(C) - X = \{v_1, v_3, \ldots, v_{2r-3}, v_{2r-1}\}$ if $k=2r+1$. It is easy to see that for $k = 2r$, any pair of elements in $V(C) - X$ correspond to a pair $\{v_{i+1}, v_{j+1}\}$ such that $v_i, v_j \in X$. Similarly for $k = 2r+1$ any pair of vertices in $V(C) - X$ correspond to a pair $\{v_{i+1}, v_{j+1}\}$, or $\{v_{i-1}, v_{j-1}\}$, unless the pair is exactly $\{v_{2r-1}, v_1\}$. 

Finally we will prove that
\begin{equation}\label{eq:noCedges}
    \mbox{$|E(H)| \leq |E'(H)| +| E(C)|$ if $k = 2r$, and $|E(H)| \leq |E'(H)| + |E(C)| + r+1$ if $k = 2r+1$.}
\end{equation} 
Let $s$ be the number of edges in $E(H) -E(C)$ that are contained in $V(C)$. Then~\eqref{eq:noCedges} is equivalent to proving $s=0$ if $k=2r$ and $s \leq r+1$ if $k =2r+1$. Suppose not, and let $f$ be such an edge.

Since $|f| =r$, $f$ is one of two types:
\begin{itemize}
    \item Type 1: $f$ contains $r-1$ vertices in $X$ and one vertex in $V(C) - X$,
    \item Type 2: $k = 2r+1$ and $f$ contains $\{v_{2r-1}, v_1\}$ and $r-2$ vertices in $X$.
\end{itemize} In particular, there are at most $|V(C)-X|$ type 1 edges, and at most ${r-1 \choose r-2} = r-1$ type 2 edges. If $k=2r+1$, since $s > (r-1)+2$, let us also assume $f$ is type $1$, and $f \notin \{\{v_1\} \cup X, \{v_{2r-1}\} \cup X\}$. 

Let $\{v_a\} = f - X$. Then $v_a \in\{v_3, v_5, \ldots, v_{2r-3}\}$ and $v_{a-1}, v_{a+1} \in X$. 
Let $C'$ be the Berge cycle of length $\ell$ obtained by replacing edge $e_a$ in $C$ with $f$. By the choice of $C$ to maximize $|\{e \in E(C): e\subseteq V(C)\}|$, we must have $e_a \subseteq V(C)$. Moreover, since $f \neq e_a$, $e_a$ contains another vertex $v_b \in V(C) - X$, and either $v_{b-1} \in A$ or $v_{b+1} \in X$ (or both).
If $v_{b+1} \in X$, then the Berge path

\[u_1, e, v_{a+1}, e_{a+1}, \ldots, v_{b}, e_a, v_a, e_{a-1}, v_{a-1}, \ldots, v_{b+1}, e', u_2\]

contains $k$ vertices. Similar holds if $v_{b-1} \in A$.  This proves~\eqref{eq:noCedges}.

Since we have shown that each vertex in $v \in V(H) -V(C)$ has degree at most $1$ in $E'(H)$, when $k=2r$ we obtain

\[|E(H)| \leq |E'(H)| + |E(C)| \leq  (n-(k-2))+(k-2),\]
which is at most $n-(k-3) + {k-3 \choose r}$ if $k-3 \geq r+1$. The remaining case $k-3=r$ (so $k=2r=r+3$, i.e., $r=3$) is proved in the Appendix.

When $k=2r+1$, 
\[|E(H)| \leq |E'(H)| + |E(C)|+ r+1 \leq (n-(k-2)) + (k-2) + r+1 = n+r+1.\]

We have $(k-3) - r =2r-2 - r \geq 1$. Suppose first that $k-3 \geq r+2$ (and hence $r \geq 4$ since $k=2r+1$). Then ${k-3 \choose r} \geq {k-3 \choose 2}$. 

We have \[n + r+1 \leq n-(k-3)+{k-3 \choose 2} = n+\frac{k-3}{2} \left((k-4)-2\right) \Leftrightarrow r+1 \leq \frac{k-3}{2} \left((k-4)-2\right).\]
Note that $(k-3)/2 = (2r-2)/2 = r-1$ and so this is equivalent to $1 + 2/(r-1) \leq k-6$. We have $k -6 \geq 2(4)+1-6 =3$, thus the inequality holds. The remaining case $k-3 = r+1$, (so $r = 3$) is proved in the Appendix. This proves~\eqref{eq:1edge1}.




Hence, there is at most one edge $e \in E'(H)$ such that $|e - V(C)|=1$. Say $e -V(C) = \{u\}$. All other edges in $E'(H)$ contain 2 or more vertices outside of $V(C)$, none of which are $u$, and these vertices have degree 1 in $E'(H)$. By Claim~\ref{cla: edges inside C},
\[|E(H)|\leq |E(C) \cup \{e\in E(H): e \subseteq V(C)\}| + |E'(H)| \]\[\leq \max\{k-2, {k-2 \choose r}\}+1+ \dfrac{n-(k-2)-1}{2}  = {k-2 \choose r} + \frac{n-(k-3)}{2}.\]
This is at most $n-(k-3)+{k-3 \choose r}$ whenever ${k-3 \choose r-1} \leq \frac{n-(k-3)}{2}$ which holds for $n \geq n_0 \geq k-3 + 2{k-3 \choose r-1}$.

\end{proof}



Given the Berge cycle $C$, for any subset of vertices $B \subseteq V(H)-V(C)$, we define \[E'(B)=\{e \in E(H)-E(C):  e \nsubseteq V(C), e - V(C) \subseteq B\}.\]

 We say $H$ is {\em $k$-connected} if its incidence bipartite graph $I=I(H)$ is $k$-connected. A cut vertex in the graph $I$ corresponds to either a vertex or a hyperedge in $H$. We call such a vertex or hyperedge a {\em cut element} in $H$. Note that  if $x$ is a cut hyperedge in $H$, then deleting the vertex $x$ from the incidence bipartite graph $I$ is equivalent to deleting the hyperedge $x$ in $H$. On the other hand, if $x$ is a cut vertex in $H$, then deleting $x$ in $I$ is equivalent to deleting the vertex in $x$ in $H$ and ``shrinking" all hyperedges that contained $x$ (rather than deleting them). This yields a non-uniform hypergraph.

We call a set of vertices $B \subseteq V(H)-V(C)$ an {\em outside connected component} if it is a maximal set such that the incidence bipartite graph on the vertex set $B \cup E'(B)$ is connected. We call a semi-path an {\em outside semi-path in $B$} if its internal elements belong to $B \cup E'(B)$. 
Similarly, we call a Berge path an {\em outside Berge path in $B$}  if its internal vertices belong to $B$ and hyperedges belong to $E'(B)$. We call a Berge cycle an {\em outside Berge cycle in $B$} if all of its vertices belong to $B$ and hyperedges belong to $E'(B)$. When $B$ is clear from the context, we will simply refer to these as outside semi-paths, outside Berge paths, and outside Berge cycles.

\begin{claim} \label{cla:bad edge}
If $B$ is an outside connected component, then there is at most one edge $e \in E'(B)$ such that $|e \cap V(C)| > \frac{\ell}{2}$.  Moreover, if $r= \ell/2+1$, then there is at most one edge $e \in E'(B)$ with $|e \cap V(C)| = \frac{\ell}{2}$.
\end{claim}

\begin{proof}
Suppose there are two edges $f,f' \in E'(B)$  with $|f \cap V(C)| \geq \frac{\ell}{2}$ and $|f' \cap V(C)| \geq \frac{\ell}{2}$. By definition, there exists a shortest outside semi-path $P'$ from $f$ to $f'$ with at least one vertex. If there exists a pair of distinct vertices $v_i \in f \cap V(C), v_j \in f' \cap V(C)$ such that $|i-j| = 1$, then replacing the edge in $C$ between $v_i$ and $v_j$ with $P'$ yields a longer Berge cycle, a contradiction. We apply Lemma~\ref{lem:intC} with $X = f \cap V(C), Y = f' \cap V(C), s = 1$, $q \geq \ell/2$. We get a contradiction except when $X = Y$ and $|X| = \ell/2$. In this case, without loss of generality $X = \{v_2, v_4, v_6, \ldots, v_{2 \lceil \ell/2 \rceil}\}$.

For the ``moreover" part of the Claim, we assume $r = \ell/2 + 1$. Since $H$ is simple, there exists distinct $w, w' \notin V(C)$ with $f - V(C) = \{w\}$ and $f'-V(C) = \{w'\}$. Let $P$ be an outside Berge path from $w$ to $w'$ (note that $f, f' \notin E(P)$). Replacing the segment $v_2, e_2, v_3, e_3, v_4$ with $v_2, f, w, P, w', f', v_4$ yields a longer Berge path than $C$, a contradiction.
\end{proof}

\begin{claim} \label{cla: outside cycle 1}
If $B$ is an outside connected component, there is no outside Berge cycle in $B$ with at least $k-\ell$ vertices.    
\end{claim}

\begin{proof}
Suppose there is a outside Berge cycle $C'=u_1, f_1, u_2, \ldots, u_{\ell'}, f_{\ell'}, u_1$ in $B$ with $\ell' \geq k-\ell$ vertices. By the connectivity of $H$, there exists an outside semi-path $P$ from $V(C') \cup E(C')$ to $V(C) \cup E(C)$. If the first element of $P$ is an edge, denote it by $f_i$. If the first element of $P$ is a vertex, denote it by $u_i$.  Similarly, if the last element of $P$ is an edge, say it is $e_j$, and if the last element of $P$ is a vertex, say it is $v_j$. We observe that \[v_{j+1}, e_{j+1}, \ldots, e_{j-1}, v_j, P, u_i, f_{i-1}, \ldots, f_{i+1}, u_{i+1}\] is a Berge path with at least $\ell'+\ell \geq k$ vertices, a contradiction.
\end{proof}

There exists a partition of $V(H) - V(C)$ into outside connected components $B_1, B_2, \ldots, B_t$ such that the edge sets $E'(B_1), E'(B_2), \ldots, E'(B_t)$ form a partition of the edges in $E(H) - E(C) - \{e: e \subseteq V(C)\}$.

If every $B_i$ satisfies $|E'(B_i)| \leq |B_i|$, then from Claim~\ref{cla: edges inside C} we obtain 

\[|E(H)| \leq \max\{\ell,{\ell \choose r}\} + \sum_{i=1}^t |E'(B_i)| \leq \max\{\ell,{\ell \choose r}\} + (n-\ell)\] 

This is less than ${k-3 \choose r} + n-(k-3)$ when $k \geq r+4$ by Claim \ref{cla: k-3}. We handle the case $k=r+3$ separately in the Appendix (see Claim~\ref{cla:kr+3}) as to not distract from the main proof. Hence from now on we consider only $k \geq r+4$.

Thus we may assume that there exists a $B_i$ satisfying $|E'(B_i)| \geq |B_i| +1$ for some $i \in[t]$. Without loss of generality, assume that it is $B_1$.
If $|B_1| = 1$, then each edge in $E'(B_1)$ contains exactly $r-1 \geq (k-3)/2 \geq \ell/2$ vertices in $V(C)$. 
Then by Claim~\ref{cla:bad edge}, we have $|E'(B_1)| \leq 1 = |B_1|$. So we assume $B_1$ has at least 2 vertices and 3 edges.

For any $B \subseteq B_1$, let $H[B]$ be the non-uniform multi-hypergraph with $V(H[B])=B$ and $E(H[B])=\{e \cap B: e\in E'(B)\}$. Thus, by this definition, we have $|E(H[B])|=|E'(B)|$. As before, we say $H[B]$ is $2$-connected if its incidence bipartite graph is $2$-connected.

\begin{claim}\label{cla: 2-conn block}
There exists a subset $B \subseteq B_1$ such that $H[B]$ is $2$-connected and $|E'(B)|\geq |B|$.  
\end{claim}

\begin{proof}
If $H[B_1]$ is $2$-connected, since $|E'(B_1)|\geq |B_1|+1$, take $B=B_1$. Now assume $H[B_1]$ is not $2$-connected. Let $I = I(H[B_1])$ denote the incidence bipartite graph of $H[B_1]$. Since the incidence bipartite graph $I$ is not $2$-connected, $I$ has a block-cut tree decomposition: a decomposition into vertex sets $D_1, D_2 \ldots, D_p$ such that each induced subgraph $I[D_i]$ is either an edge or a maximal $2$-connected subgraph, and $|D_i \cap D_j| \leq 1$ whenever $i \neq j$. It is well known that if $I$ is not $2$-connected, then there exists a leaf-block: a $D_i$ that is incident to at most $1$ cut element.

For each block $D_i$, let $H_{D_i}$ denote the subhypergraph with $V(H_{D_i}) = D_i \cap V(H)$ and $E(H_{D_i}) = \{e \cap B_1: e \cap B_1 \subseteq V(H_{D_i})\}$.
Note that any cut vertex incident to $H_{D_i}$ is contained in $V(H_{D_i})$ while any cut hyperedge incident to $H_{D_i}$ connects to vertices in different blocks, and hence does not belong to $E(H_{D_i})$.

We run a pruning process as follows.
We pick an arbitrary leaf block, say $D_i$. If $|E(H_{D_i})| \geq |V(H_{D_i})|$, then $H_{D_i}$ is our desired block. We stop and return $D_i$.

Otherwise, we have $|E(H_{D_i})| \leq |V(H_{D_i})|-1$. Since $H_{D_i}$ is a leaf block, it is adjacent to at most one cut element. Let $H^0=H[B_1]$. Let $H^1$ be the returning hypergraph obtained from $H^0$ by pruning $D_i$ so that $V(H^1)=B_1 - V_0$ and $E(H^1)=E(H^0) - E_0$, where:

\begin{itemize}
    \item If $H_{D_i}$ is a connected component of $H[B_1]$, then $V_0 = V(H_{D_i}), E_0 = E(H_{D_i})$.
    \item If $H_{D_i}$ is incident to a cut vertex $x$, then $V_0=V(H_{D_i}) - \{x\}$ and $E_0=E(H_{D_i})$.

    \item If $H_{D_i}$ is incident to a cut hyperedge $e$, then $V_0=V(H_{D_i})$ and $E_0=E(H_{D_i})\cup \{e\}$.
\end{itemize}

Since $|E(H_{D_i})| \leq |V(H_{D_i})|-1$, we observe that in each of the three cases above, the number of edges we remove is at most the number of vertices removed.

We repeat the pruning process on $H^1$. Note that pruning a leaf block may have created more blocks (e.g., by the deletion of cut hyperedges), and some blocks may have become leaf blocks in the new hypergraph. Let $H^{j+1}$ be the returning hypergraph after the pruning process applied to the block-cut tree decomposition of $H^j$. We continue this process until we either return some block $D^*$ of the current hypergraph with $|E(H_{D^*})| \geq |V(H_{D^*})|$, or no leaf block satisfies this and we prune all of $H[B_1]$ until no vertices or edges remain.

If all of $H[B_1]$ is pruned, then the at each step, the number of edges removed at each step is at most the number of vertices removed. Hence $|E(H[B_1])| \leq |B_1|$, a contradiction.



So suppose $D^*$ is the desired block returned by the pruning process. Let $B=V(H_{D^*})$. We have $H[B]$ is $2$-connected and $|E'(B)| \geq |B|$.
\end{proof}

Let $B \subseteq B_1$ be the subset such that $H[B]$ is $2$-connected and $|E'(B)| \geq |B|$ whose existence is guaranteed by Claim \ref{cla: 2-conn block}. Let $e^* \in E'(B)$ such that $|e^* \cap V(C)|$ is maximized. Let $H^*[B]$ be the (possibly) non-uniform multi-hypergraph with $V(H^*[B])=B$ and $E(H^*[B])=\{e \cap B: e\in E'(B)-\{e^*\}\}$. Let $F_q(B)=\{f \in E'(B): |f \cap V(C)|\geq q\}$ denote the edges in $E'(B)$ that contain at most $r-q$ vertices in $B$.

\begin{claim} \label{cla:F1}
 $|e^* \cap V(C)|\geq 1$. 
\end{claim}

\begin{proof}
Suppose $|e^* \cap V(C)|=0$. Then $H[B]$ is $r$-uniform (in particular $E(H[B]) \subseteq E(H)$) and $|E(H[B])| = |E'(B)| \geq |B|$. By Corollary \ref{cor:multi cycle}, $H[B]$ contains a Berge cycle of length at least $r$ which can be lifted to an outside Berge cycle of the same length.
Let $C'=u_1, f_1, \ldots, u_{\ell'}, f_{\ell'}, u_1$ be the outside Berge cycle in $B$ where $\ell' \geq r$.

By Claim \ref{cla: outside cycle 1},  $\ell \leq k-\ell'-1\leq r$. Since $\ell \geq r$, we have $\ell=\ell'=r$. Since $H$ is connected, there exists a semi-path from $V(C') \cup E(C')$ to $V(C) \cup E(C)$. Let $P$ be the shortest such semi-path. If the first element of $P$ is a vertex, we may assume it is $u_j$. If the first element of $P$ is an edge, we may assume it is $f_j$. If the last element of $P$ is a vertex, we may assume it is $v_i$ and if it is an edge, we may assume it is $e_i$. If $P$ contains a vertex from $B- V(C')$, then 
\[u_{j+1}, f_{j+1}, \ldots, f_{j-1}, u_j, P, v_i, e_{i-1}, \ldots, e_{i+1}, v_{i+1}\]
is a Berge path with at least $r+r+1 \geq k$ vertices, a contradiction. Thus, $P$ contains no vertex from $B - V(C')$ and we may assume $P=u_j,e_i$. Since $H$ is simple and $\ell'=r=|V(C')|$, either $f_j$ or $f_{j-1}$ must contain a vertex $u \in B - V(C')$. Without loss of generality, assume $u \in f_j$. Then,
\[u, f_j, u_{j+1}, f_{j+1}, \ldots, f_{j-1}, u_j, e_i, v_i, e_{i-1}, \ldots, e_{i+1}, v_{i+1}\]
is a Berge path with $r+r+1\geq k$ vertices, a contradiction. 
\end{proof}

It follows that $H[B]$ is not $r$-uniform since $|e^* \cap B| \leq r-1$. 

Recall that for integers $a,b$, a hypergraph $G$ is called an {\em $(a,b)$-block tree} if its 2-shadow $\partial_2G$ is connected and every $2$-connected block of $\partial_2G$ has exactly $a$ vertices and induces exactly $b$ edges in $G$. In particular, if $G$ is $a$-uniform, then each block induces $b$ multi-edges.

\begin{claim}\label{cla: block tree}
For $0 \leq t  \leq r-2$, $H^*[B]$ is not a $(r-t )$-uniform, $(r-t ,r-t -1)$-block tree. 
\end{claim}

\begin{proof}
    Suppose $H^*[B]$ is a $(r-t )$-uniform, $(r-t ,r-t -1)$-block tree. 
If $t =0$, then since $H$ is $r$-uniform with no multi-edges, $H^*[B]$ also has no multi-edges. Thus $H^*[B]$ is not an $(r,r-1)$-block tree.
    
    Now let $t  \geq 1$. There is a $(r-t )$-subset $W \subseteq B$ with multiplicity $r-t -1$ in $E(H^*[B])$. As $H$ is an $r$-uniform hypergraph, these edges correspond to $r-t -1$ edges in $H$ that each contain $W$ and $t $ vertices in $V(C)$. Let these edges of $H$ be $h_1, h_2, \ldots, h_{r-t -1}$. Moreover, since $H$ contains no multi-edges, each $h_i\cap V(C)$ is a distinct $t $-set. By the choice of $e^*$ to maximize $|e^* \cap V(C)|$, we also have $|e^* \cap V(C)| \geq t  \geq 1$.

    First observe that
    \begin{equation}\label{eq:2r-2}
        \text{\parbox{0.8\linewidth}{$\ell \leq 2r-2$, and if $\ell=2r-2$, then any outside semi-paths from $B_1$ to $F_1(B_1)$ contain at most 2 vertices.} }
    \end{equation}
    Indeed, we can find an outside semi-path from $W$ to $e^*$ with at least $2$ vertices in $W$. Combining this semi-path with a spanning Berge path of $C$ yields a Berge path with at least $2 + \ell \leq 2r$ vertices. Similarly, if an outside semi-path from $B_1$ to $F_1(B_1)$ contains $3$ or more vertices and $\ell= 2r-2$, we obtain a Berge path on at least $2r-2 + 3 \geq k$ vertices. This proves~\eqref{eq:2r-2}.

    {\bf Case 1}: $r-t  \geq 3$. Let $P$ be a shortest outside semi-path in $H$ from $W$ to $e^*$. Without loss of generality, $h_1 \cap V(C) \neq e^* \cap V(C)$ and we can extend $P$ to an outside semi-path $P'$ from $h_1$ to $e^*$ containing edges $\{h_1, h_2, \ldots, h_{r-t -1}, e^*\}$ and $r-t -1$ vertices in $W$.   

    If there exists distinct $v_i \in h_1\cap V(C)$ and $v_j \in e^*\cap V(C)$ with $|i-j| \leq r-t -1$, then replacing a shorter segment in $C$ from $v_i$ to $v_j$ with the Berge path $v_i, h_1, P', e^*, v_j$ yields a Berge cycle of length $\ell - (|i-j|-1) + (r-t -1) > \ell$, contradicting the choice of $C$.

    Thus by Lemma~\ref{lem:intC} with $X = h_1 \cap V(C), Y = e^* \cap V(C)$, $q = t $, $s = r-t -1$, if $X \not\subseteq Y$
   \[\ell \geq 2t  + 2(r-t -1) \geq 2r-2.\]

   If $X \subsetneq Y$, then for any $ v_i \in Y - X$, the elements in $X \cup \{v_i\}$ are at least $r-t $ apart in $C$, and so \[\ell \geq (t +1)(r-t ).\] This function is minimized when $t  \in \{1,r-2\}$, which gives $\ell \geq 2r-2$. 

    By combining $P$ with a Berge path in the $(r-t , r-t -1)$-block tree, we obtain an outside semi-path from $W$ to $e^*$ which covers all vertices of $W$. This semi-path contradicts~\eqref{eq:2r-2}.

   {\bf Case 2}: $r-t  = 2$.  So $t  =r-2$. We have $|e^*\cap V(C)| \geq t =r-2$. Recall $h_1$ is the edge that contains $W$ in $B$ and $r-2$ vertices in $C$. Let $h_1 \cap B=\{u,u'\} = W$.

   {\em Case 2.1}: $\ell = 2r-2$. 
   If the $(2,1)$-block tree $H^*[B]$ contains at least $2$ blocks, then since $H[B]$ is $2$-connected, when we add the edge $e^*$ back, we can find an outside Berge cycle in $B$ with at least 3 vertices. This contradicts Claim~\ref{cla: outside cycle 1}. So we assume that $H^*[B]$ has $2$ vertices $\{u,u'\}$ and a single edge. Then $H[B]$ consists of a pair of multi-edges with vertices $\{u, u'\}$, which implies $B =\{u, u'\}$ and $E'(B) = \{e^*, h_1\}$. Here we use the fact that $H[B]$ is $2$-connected.

   Recall that $|E'(B_1)| \geq |B_1|+1 > |E'(B)|$. There exists another edge $e \in E'(B_1) - \{e^*, h_1\}$. Since $e \notin E'(B)$, $e$ contains a vertex in $B_1 - B = B_1 - \{u,u'\}$. We can find an outside semi-path from one of these vertices, say $w$, to $e^* \in F_1(B_1)$ which contains at least $3$ vertices, $w, u, u'$, in $B_1$. This contradicts~\eqref{eq:2r-2}.

{\em Case 2.2}: $\ell \leq 2r-3$. 
   By the $2$-connectivity of $H[B]$, there exists an outside Berge cycle of $H$ in $B$ using only the edges corresponding to the $(2,1)$-block tree and $e^*$. Note that all of these edges belong to $F_{r-2}(B)$. Let $C'= u_1, f_1, \ldots, u_{\ell'}f_{\ell'}, u_1$ be a longest such Berge cycle. If $|V(C')|=\ell' \geq 3$, then for every $i$, we can take the longer segment of $C'$ between $f_i$ and $f_{i+1}$ which has $\ell'-1$ vertices. Call this semi-path $P_i$.
   
Setting $X = f_i \cap V(C), Y=f_{i+1} \cap V(C)$, we have $|i' - j'| \geq |V(P_i)|+1 = \ell'$ for any distinct $v_{i'} \in X, v_{j'} \in Y$. 

By Lemma~\ref{lem:intC}, $\ell \geq \min\{(r-2)\ell', 2(r-2) + 2(2)\} \geq 2r-3$ with equality only if $r = 3$, $\ell'= 3$, $\ell = (r-2)3 = 3$, and $X = Y$. Moreover, this holds for all $P_i$, so we may assume all $f_i \cap V(C)$ are the same. Say $f_i \cap V(C) = \{v_1\}$.

Since $r = |V(C)|$, one of $e_1$ or $e_\ell$, say $e_1$, contains a vertex $w$ outside of $V(C)$. If this vertex belongs to $V(C')$, say $w = u_i$, then replacing the segment $v_1, e_1, v_2$ in $C$ with $v_1, f_{i}, u_{i+1}, \ldots, f_{i-1}, u_i, e_1, v_2$ yields a longer Berge cycle than $C$. Otherwise, if $w \notin V(C) \cup V(C')$, we take the Berge path $u_1, f_1, \ldots, f_{\ell'}, v_1,e_\ell, v_{\ell}, \ldots, v_2, e_1, w$ contains $\ell + \ell' + 1 = 2r+ 1 = k$ vertices, a contradiction.

It follows that the longest Berge cycle in $H[B]$ has length $2$, and so the $(2,1)$-block tree $H^*[B]$ contains a single block. Since  $H[B]$ is $2$-connected, as in the previous case, $B = \{u,u'\}$, $E(H[B])$ contains only the edge $\{u,u'\}$ with multiplicity $2$, and $E'(B) = \{h_1, e^*\}$. As $H$ contains no multi-edges, $h_1 \cap V(C) \neq e^* \cap V(C)$. By Lemma~\ref{lem:intC} with $s = 1$, $q = r-2$, we obtain $\ell \geq 2(r-2) + 2(1) > 2r-3$.
\end{proof}

\begin{claim} \label{cla:2-conn}
Let $2\leq q \leq r-2$. Suppose $|F_q(B)| \geq 2$. Then every outside Berge cycle $C'$ in $B$ of length at least $r-q$ contains all edges in $F_q(B)$.
\end{claim}

\begin{proof}
Suppose on the contrary that there is an outside Berge cycle $C'=u_1, f_1, u_2, \ldots, f_{\ell'-1}, u_{\ell'}, f_{\ell'}, u_1$ on $\ell' \geq r-q$ vertices in $B$ that does not use some edge $e \in F_q(B)$.

First we will show the following.
\begin{equation}\label{outsidecyclepath}
    \text{\parbox{0.8\linewidth}{There exists an outside Berge semi-path from $e$ to $F_q(B) - \{e\}$ with at least $\lceil (r-q+1)/2 \rceil$ vertices containing a segment of $C'$ with at least $\lceil(r-q)/2\rceil$ vertices.}}
\end{equation}
Since $H[B]$ is $2$-connected, by Menger's theorem \cite{D} applied to the $2$-connected incidence bipartite graph of $I(H[B])$ of $H[B]$, there exist $2$ vertex- and edge-disjoint outside semi-paths from $V(C') \cup E(C')$ to $\{e,f\}$ where $f \in F_q(B)$. Choose $P_1$, $P_2$ to be a pair of such disjoint semi-paths such that $P_1$ ends in $e$ and $P_2$ ends in $f$ such that each $P_i$ intersects $V(C') \cup E(C')$ in exactly one element (namely the first element) and $|V(P_1)| + |V(P_2)|$ is maximized. Note that since $e \notin E(C')$, $|V(P_1)| \geq 1$.

Say $w_i = (V(C') \cup E(C')) \cap P_i$ for $i \in \{1,2\}$. Then let $P'$ be the longer semi-path in $C'$ from $w_1$ to $w_2$. 
Observe that $|V(P')| \geq \lceil |V(C')|/2 \rceil$ and if $\{w_1, w_2\} \cap V(C') \neq \emptyset$, then $|V(P')| \geq \lceil (|V(C')|+1)/2 \rceil \geq \lceil (r-q+1)/2 \rceil$. The union $P_1 \cup P' \cup P_2$ is an outside semi-path with at least $\lceil (r-q+1)/2 \rceil$ vertices, as desired. This proves~\eqref{outsidecyclepath}.



Now let $P$ be a longest outside semi-path satisfying~\eqref{outsidecyclepath}.  If there exists distinct vertices $v_i \in e \cap V(C)$ and $v_j \in f \cap V(C)$ such that $|i-j| \leq |V(P)|$, then replacing the shorter segment in $V(C)$ from $v_i$ to $v_j$ with $P$ yields a longer Berge cycle than $C$, a contradiction. Hence for all such $v_i, v_j, |i-j| \geq |V(P)|+1 \geq \lceil (r-q+1)/2 \rceil+1$. 

If $|V(C)| + |V(C')| = \ell + \ell' \geq k$, then we obtain a contradiction by Claim~\ref{cla: outside cycle 1}. In particular, $\ell < k-\ell' \leq 2r+1 - (r-q) = r+q+1$.

By Lemma~\ref{lem:intC} with $X = e \cap V(C)$ and $Y = f \cap V(C)$, if $|X\cap Y| \leq q-1$, then

\[\ell \geq 2q +  2  \lceil (r-q+1)/2 \rceil + |X \cap Y| \lceil (r-q+1)/2-1 \rceil\geq 2q + r-q+1 =r+q+1,\]
a contradiction.

Now suppose $|X \cap Y| \geq q$. Instead Lemma~\ref{lem:intC} gives $\ell \geq q(\lceil (r-q+1)/2 \rceil + 1) = q\lceil(r-q+3)/2 \rceil.$ 

{\bf Case 1}: $q \geq 3$ or $q=2$ and $q,r$ have the same parity.

Suppose first that $q$ and $r$ have the same parity. Then 
\begin{eqnarray*}
    \ell & \geq q \left(\frac{r-q+4}{2} \right) & \geq r + q\\
     \Leftrightarrow & (q-2) r &\geq  q^2-2q \\
     \Leftarrow & (q-2)(q+2) & \geq q^2-2q\\
     \Leftrightarrow & 2q & \geq 4.
\end{eqnarray*}
    The last inequality is true because we assume $q \geq 2$. 

If $q$ and $r$ have opposite parities and $q \geq 3$, then $r \geq q+3$. We instead get

\begin{eqnarray*}
    \ell & \geq q \left(\frac{r-q+3}{2} \right) & \geq r + q \\
     \Leftrightarrow & (q-2) r &\geq  q^2-q \\
     \Leftarrow & (q-2)(q+3) & \geq q^2-q\\
     \Leftrightarrow & 2q & \geq 6.
\end{eqnarray*}

Thus $\ell = r+q$. Note that we have equality only if $|V(P)| = \lceil (r-q+1)/2 \rceil$. Also, $\ell' = r-q$. Let $P'$ be a shortest outside semi-path from $V(C') \cup E(C')$ to $e \in F_q(B)$. If $P'$ contains an internal vertex, then there exists a Berge path with $r-q$ vertices in $C'$, $r+q$ vertices in $C$ and at least one vertex in $P'$. That is, we obtain a Berge path of length at least $2r+1 \geq k$, a contradiction.

Otherwise, $P' = u_i, e$ for some $u_i \in V(C')$. Suppose the edge $f_i \in E(C')$ contains a vertex $w \notin V(C) \cup V(C')$. For any $v_j \in e \cap V(C)$, we obtain the Berge path \[w, f_i, u_{i+1}, f_{i+1}, \ldots, u_{i-1}, f_{i-1}, u_i, e, v_j, e_j, v_{j+1}, \ldots, v_{j-1}\] which has $1+(r-q)+(r+q) \geq k$ vertices. 

So $f_i$ has at most $r-q = \ell'$ vertices in $B$, and hence $f_i \in F_q(B)$.  The outside semi-path $f_i, u_{i+1}, f_{i+1}, \\
\ldots, u_i, e$ has $r-q > (r-q + 1)/2$ vertices. Then this is a better choice than $P$ of a semi-path satisfying~\eqref{outsidecyclepath}, a contradiction. 

{\bf Case 2}: $q =2$ and $r$ is odd. By Lemma~\ref{lem:intC}, $r+q \geq \ell \geq 2 (|V(P)|+1)$, so $|V(P)| \leq (r+q-3)/2 = (r-1)/2$. If $\ell = r+q$, then as in the previous case, we can find a Berge path with $k$ vertices.
Also by Lemma~\ref{lem:intC}, \[\ell \geq 2\left(\frac{r-q+1}{2} + 1 \right) = r-q+3 = r+q-1,\] when $q = 2$. Here we have equality only if $X = Y = \{v_i, v_{i+(r-1)/2 +1}\}$ for some $i$ and $|V(P)| = (r-1)/2$ and $\ell=r+1$. By Claim \ref{cla: outside cycle 1}, we have $\ell'\leq k-\ell-1 \leq r-1$. Thus $\ell' \in \{r-2, r-1\}$. Let $v_i \in e \cap V(C)$. First we will show that $e_i \cap B =\emptyset$. If not, then by the definition of outside connected components, there exists an outside semi-path $P''$ from $e$ to the vertices of $e_i\cap B$ which contains at least one vertex in $B$. Then replacing the segment $v_i, e_i, v_{i+1}$ on $C$ by $v_i, e, P'', e_{i}, v_{i+1}$, we obtain a Berge cycle longer than $C$, contradicting the choice of $C$. Hence, $e_i \cap B =\emptyset$. Next, we will show that $e_{i} \in E(C)$ contains a vertex $w' \notin V(C) \cup V(C')$. Otherwise if $e_i \subseteq V(C) \cup V(C')$, then $e_i \subseteq V(C)$ since $V(C') \subseteq B$. As $|V(C)| = r+1$ and $|e_i| = r$, $e_i$ contains at least one of $v_{i+(r-1)/2 +2}$ and $v_{i+(r-1)/2 +3}$. If $v_{i+(r-1)/2 +3} \in e_i$, we obtain the Berge cycle \[f, P, e, v_i, e_{i-1}, v_{i-1}, \ldots, v_{i+(r-1)/2 +3}, e_i, v_{i+1}, e_{i+1}, \ldots, v_{i+(r+1)/2+1}, f\] which contains $(r-1)/2 + (r+1)-1 > |V(C)|$ vertices, contradicting the choice of $C$. The case of $v_{i+(r-1)/2 +2} \in e_i$ is similar. Hence, $e_{i} \in E(C)$ contains a vertex $w' \notin V(C) \cup V(C')$.

Let $P'''$ be the shortest semi-path from $e$ to $V(C') \cup E(C')$. Assume the last element of $P'''$ is $u_1$ if it is a vertex and $f_{\ell'}$ if it is an edge. If $\ell'=r-1$, then 
\[w', e_i, v_{i+1}, e_{i+1}, \ldots, e_{i-1}, v_i, e, P''', u_1, f_1, \ldots, u_{\ell'}\]
is a Berge path with at least $(r-1)+(r+1)+1=2r+1 \geq k$ vertices, a contradiction. Thus, we must have $\ell'=r-2$. If $P'''$ contains a vertex from $B - V(C')$, then there exists a Berge path with $r-2$ vertices in $C'$, at least one vertex from $P'''$, $r+1$ vertices in $C$, and the last vertex $w'$. That is, a Berge path with at least $2r+1 \geq k$ vertices, a contradiction. Hence, we have $P'''=e, u_1$ and we may assume $P = e, u_1, f_1, u_2, \ldots, u_{(r-1)/2}, f_{(r-1)/2}$ where $f = f_{(r-1)/2} \in E(C')$ 

If $f_{\ell'} \in F_2(B)$, then $e, u_1, f_1, \ldots, f_{\ell'}$ is a better choice than $P$ for~\eqref{outsidecyclepath}. Therefore $f_{\ell'}$ contains at least $r-1$ vertices in $B$ and hence some $w \notin V(C) \cup V(C')$. Then,
\[w, f_{\ell'}, u_{\ell'}, f_{\ell'-1},\ldots, u_1, e, v_i, e_{i-1}, v_{i-1}, \ldots, v_{i+1}, e_i, w' \] is a Berge path with $(r-2)+(r+1)+2=2r+1 \geq k$ vertices if $w \ne w'$ or a Berge cycle with $(r-2)+(r+1)+1=2r > \ell$ vertices if $w=w'$. 
\end{proof}

Let $q$ be the largest integer such that $|F_q(B)| \geq 2$ and note $e^* \in F_q(B)$. Since $r-1 \geq (k-3)/2\geq \ell/2$, $0\leq q \leq r-2$ by Claim \ref{cla:bad edge}. 
Since $H^*[B]$ does not contain the edge corresponding to $e^*$, each edge of $H^*[B]$ has cardinality at least $r-q$. Moreover $|E'(B)|\geq |B|$ implies $|E(H^*[B])| \geq |B|-1$. By Corollary \ref{cor:multi cycle} and Claim~\ref{cla: block tree}, $H^*[B]$ contains a Berge cycle with at least $r-q$ vertices which can be lifted to an outside Berge cycle in $B$ of $H$ with the same length that omits the edge $e^* \in F_q(B)$. 

Among all such outside Berge cycles in $B$ that omit $e^*$, let $C'$ be the longest one. Then $|V(C')| \geq r-q$, and by Claim \ref{cla: outside cycle 1}, $|V(C')|\leq k-\ell-1 \leq 2r+1-r-1=r$. Say $|V(C')| = r-t$ where $t \leq q$.

We get a contradiction to Claim~\ref{cla:2-conn}, unless $t \leq q \leq 1$.  To complete the proof of Theorem~\ref{thm:main}, then it suffices to show 

\begin{equation}\label{t2}t \geq 2.\end{equation}

    Suppose $t \leq 1$, that is $\ell':=|V(C')| \in \{r,r-1\}$. From Claim~\ref{cla: outside cycle 1}, $\ell = |V(C)| \leq k-\ell'-1 \leq 2r- \ell' \leq r+1$. Thus, $\ell \in \{r, r+1\}$. Let $C' = u_1,f_1, u_2, \ldots, u_{\ell'}, f_{\ell'},u_1$, and let $P$ be a shortest outside semi-path from $V(C') \cup E(C')$ to $e^*$. If the first element of $P$ is a vertex, without loss of generality assume it is $u_1$, and if the first element is an edge, assume it is $f_{\ell'}$. We may also assume by symmetry that $v_1 \in e^*$. 

{\bf Case 1}:  $\ell = |V(C)| = r$. As $H$ is a simple, $r$-uniform hypergraph, one of $e_1$ and $e_{\ell}$, say $e_{\ell}$, contains a vertex $z$ outside of $V(C)$. If $z \in B$, then there exists a semi-path $P'$ from $z$ to $e^*$.  Replacing the segment $v_{\ell}, e_{\ell}, v_1$ on $C$ by $v_{\ell}, e_{\ell}, z, P', e^*, v_1$, we obtain a Berge cycle longer than $C$, a contradiction. Thus $z \notin B$; in particular it does not belong to $V(C')\cup V(C) \cup V(P)$.

Set
   \[Q = u_{\ell'}, f_{\ell'-1}, \ldots, u_1, P, e^*, v_1, e_1, \ldots, v_{\ell}, e_{\ell}, z.\]
    
    If $(\ell, \ell') = (r,r)$ or if $(\ell, \ell')= (r,r-1)$ and $P$ contains an interior vertex, then $Q$ is a Berge path with at least $r+r+1 \geq k$ vertices, a contradiction. Then we assume $\ell' = r-1$ and $P=u_1,e^*$. We must have $f_i \cap V(C) \subseteq \{v_1\}$ for $i \in \{1,\ell'\}$: indeed if say $f_1$ contains some $v_j \in V(C)$ with $j \neq 1$, then replacing the shorter segment in $C$ from $v_1$ to $v_j$ with $v_1, P, u_1, f_{\ell'}, u_{\ell'}, f_{\ell'-1}, \ldots, u_2, f_1, v_j$ yields a Berge path with at least $\ell-(\ell-1)/2 + (r-1) >\ell$, contradicting the choice of $C$ as a longest Berge cycle. The case for $f_{\ell'}$ is similar.
    
    As $f_1, f_{\ell'}$ are distinct $r$-subsets, one of these edges contains a vertex $z'\notin V(C) \cup V(C')$. If $z' \in f_{\ell'}$, then we append $z', f_{\ell'}$ to the beginning of $Q$ to obtain a Berge path on $2r+1 \geq k$ vertices or Berge cycle $2r > \ell$ vertices if $z' = z$.  
    
    Similarly if $z' \in f_1$, we instead take $z', f_1, u_2, f_2, \ldots, f_{\ell'}, u_1, e^*, v_1, e_1, \ldots, v_{\ell}, e_{\ell}, z$ of the same length.

    {\bf Case 2}: $\ell = r+1$. The only case to check is $(\ell, \ell') = (r+1, r-1)$. Then we have $P = u_1, e^*$ and as in the previous case, one of $f_1$ or $f_{\ell'}$ contains a vertex $z' \notin V(C) \cup V(C')$. Say $z' \in f_{\ell'}$. Then the Berge path 
    \[z', f_{\ell'}, u_{\ell'}, f_{\ell'-1}, \ldots, u_1, P, e^*, v_1, e_1, \ldots, v_{\ell}\] contains at least $1 + (r-1) + (r+1) \geq k$ vertices, a contradiction.

    This proves~\eqref{t2} and completes the proof of Theorem~\ref{thm:main}.

\qed

\section{Concluding Remarks}
In~\cite{FKL}, the second author with F\"uredi and Kostochka proved that if $r, k$ are integers with $k \geq 4r \geq 12$ and $n$ is sufficiently large, then every $2$-connected, $n$-vertex, $r$-uniform hypergraph with more than ${\lceil (k+1)/2 \rceil \choose r} + {\lfloor (k-1)/2 \rfloor \choose r-1}(n-\lceil (k+1)/2 \rceil)$ edges contains a Berge cycle of length at least $k$. Later, Zhao, Gerbner, and Zhou~\cite{ZGZ} extended the bounds on $r,k$ to $k \geq 2r+2 \geq 8$.
As a corollary of their result, they obtained exact Tur\'an number of Berge $P_k$ in connected hypergraphs when $k \geq 2r+2$.

On the other hand, no exact results are known for the Tur\'an number of Berge $C_{\geq k}$ in $r$-uniform, $2$-connected hypergraphs where $k \leq 2r+1$. 
Here we use $C_{\geq k}$ to denote the set of all cycles with length at least $k$.

The case in which the hypergraph $H$ is allowed to contain multi-edges and the $r$-uniform condition is replaced with {\em lower rank $r$} (i.e., all edges contain at least $r$ vertices) is equivalent to an extremal problem for its incidence bipartite graph.

For integers $a,b,r$, define $G(a,b,r)$ to be the set of bipartite graphs $A \cup B$ with $|A| = a, |B| = b$ and $d(x) \geq r$ for all $x \in A$.
Jackson~\cite{Jackson} conjectured the following.

\begin{conj}[Jackson~\cite{Jackson}]
     Suppose $a \geq r$ and $G \in G(a,b,r)$ is a $2$-connected graph with  
     \[b \leq \left\lfloor \frac{2(a - \epsilon_r)}{r-1-\epsilon_r} \right\rfloor (r-2) + 1,\]

     where $\epsilon_r = 1$ if $r$ is even and $\epsilon_r=0$ if $r$ is odd. 
Then $G$ has a cycle of length at least $2r$.
\end{conj}
In particular if $r$ is odd, $H$ is a $2$-connected, $n$-vertex (multi)hypergraph with lower rank $r$, and $r-1$ divides $2|E(H)|$, then Jackson's Conjecture applied to the incidence bipartite graph $I(H)$ states that if
$|E(H)| \geq \frac{(n-1)(r-1)}{2(r-2)}$, then $H$ contains a Berge cycle of length at least $r$. Moreover, we can construct a multi-hypergraph on $\frac{(n-2)(r-1)}{2(r-2)}$ edges with no Berge cycle of length at least $r$: let $e_1, e_2, \ldots, e_{(n-2)/(r-2)}$ be $r$-subsets of $[n]$ in which every pair intersect in exactly $\{1,2\}$. Then let $H$ be obtained by taking $(r-1)/2$ copies of each $e_i$. 

We now define two families of $r$-uniform hypergraphs, the first one being similar to this construction.

\begin{construction}
    For $k \leq 2r+1$, let $(r-1) \mid (n-2)$, say $n = t(r-1) +2$. The family $\mathcal H_2$ consists of all $n$-vertex $r$-graphs $H$ where $V(H)$ is partitioned into sets $\{x,y\} \cup V_1 \cup \ldots, \cup V_t$ with $|V_i| = r-1$, and for each $i \in [t]$, $E(H)$ contains $\lfloor \frac{k-1}{2} \rfloor$ edges in $\{x,y\} \cup V_i$. Moreover, if $k$ is even, then we add one additional edge in $\{x,y\} \cup V_1$.
\end{construction}

\begin{construction}
    For $k = 2r+1$, let  $n = t (k-3)/2+2$. The family $\mathcal H_3$ consists of all $n$-vertex $r$-graphs $H$ where $V(H)$ is partitioned into sets $\{x,y\} \cup V_1 \cup\ldots \cup V_t$ with $|V_i| =(k-3)/2 $ for $i \in[t]$, and each $\{x,y\} \cup V_i$ induces a clique. 
\end{construction}

In both constructions, a longest Berge cycle intersects at most $2$ $V_i$'s and hence has length at most $k-1$.
We pose the following question.

{\bf Question}: For $r \geq 3$ and $k \leq 2r+1$, do the hypergraphs in $\mathcal H_2 \cup \mathcal H_3$ maximize the number of edges over all $r$-uniform, $n$-vertex, $2$-connected hypergraphs without Berge cycles of length $k$ or longer?

\bigskip

{\bf AI Disclosure}: 
No AI-assistance of any kind was used in the proof or preparation of this paper.

\newpage
\section{Appendix: Proofs for case $k=r+3$.}

We complete the proof of~\eqref{eq:1edge1} for the cases $r=3$, $\ell=k-2$, $k \in \{2r,2r+1\}$, and $k -3=r$ when $k = 2r$, $k-3=r+1$ when $k = 2r+1$.

\begin{proof}[Remainder of the Proof~\eqref{eq:1edge1} for $r=3$.] We assume there are two edges $e, e'$ with $|e -V(C)| =|e'-V(C)| = 1$, and $e \cap V(C) = e'\cap V(C) = \{v_2, v_4\}$. 
Say $e -V(C) = \{u_1\}, e'-V(C) = \{u_2\}$, where $u_1 \neq u_2$. Let $C = v_1, e_1, \ldots, v_\ell, e_\ell, v_1$. 

{\bf Case 1}: $k = 2r$. Here $\ell = k-2 = 4$.
If the edge $e_1$ contains a vertex $u \notin V(C)$, then without loss of generality assume $u \neq u_1$, and we can take the Berge path on $k$ vertices $u_1, e, v_2, e_2, v_3, e_3, v_4, e_4, v_1, e_1, u$. Similarly, one can show that all edges $e_i \in E(C)$ are contained in $\{v_1, v_2, v_3, v_4\}$. 

If we can rearrange the edges of $E(C)$ to obtain a Berge path $P$ from $v_2$ to $v_4$ containing $\{v_1, v_2, v_3, v_4\}$, then $u_1,e, P, e', u_2$ is a Berge path on $k$ vertices.

Therefore both $v_2, e_1, v_1, e^*, v_3,e_3, v_4$ where $e^* \in \{e_2, e_4\}$ and $v_2,e_2, v_3,f^*, v_1,e_4, v_4$ where $f^* \in \{e_1, e_3\}$ are not Berge paths. In particular, from $e^*$, we obtain $v_1 \notin e_2, v_3 \notin e_4$, and hence $e_2 = \{v_2,v_3, v_4\}, e_4 = \{v_4, v_1, v_2\}$. Similarly from $f^*$ we obtain $v_3 \notin e_1, v_1 \notin e_3$, so $e_1 = \{v_1, v_2, v_4\}$ and $e_3 = \{v_3, v_4, v_2\}$. Then $e_1 = e_4$, contradicting that $H$ contains no multi-edges.

{\bf Case 2}: $k = 2r+1$. Here $\ell = k-2 = 5$. Similar to Case 1, one can show that $e_1, e_2, e_3, e_4 \subseteq V(C)$, and there is no Berge path $P$ from $v_2$ to $v_4$ containing $\{v_1, v_2, v_3, v_4, v_5\}$ with edges in $E(C)$.

Then $v_2, e_1, v_1, e_5, v_5, e^*, v_3, e_3, v_4$ with $e^* \in \{e_2, e_4\}$ and $v_2,e_2, v_3, f^*, v_1, e_5, v_5, e_4, v_4$ with $f^* \in \{e_1, e_3\}$ are not Berge paths.
 It follows that $v_5 \notin e_2, v_3 \notin e_4$, $v_3 \notin e_1$, and $v_1 \notin e_3$. 

We have $e_2 = \{v_2, v_3, v_1\}$ or $e_2 = \{v_2, v_3, v_4\}$, and $e_3 = \{v_3, v_4, v_5\}$ or $e_3 = \{v_3,v_4, v_2\}$. Without loss of generality, assume $e_2 \neq \{v_2, v_3, v_4\}$. If $e_3 = \{v_2, v_3, v_4\}$, then we take  $P= v_2, e_3, v_3, e_2, v_1, e_5, v_5, e_4, v_4$. If $e_3 = \{v_3,v_4, v_5\}$, then $P=v_2, e_1, v_1, e_2, v_3, e_3, v_5, e_4, v_4$. In both cases, we obtain a contradiction. This completes the proof of~\eqref{eq:1edge1}.
\end{proof}

Here we also prove that a counterexample $H$ cannot have $k=r+3$, assuming Claims~\ref{cla: edges inside C} through~\ref{cla: outside cycle 1}.

\begin{claim} \label{cla:kr+3}
$k \geq r+4$.
\end{claim}

\begin{proof}
Suppose on the contrary that $k=r+3$. By Claim \ref{cla: k-3}, we have $k-3=r\leq \ell \leq k-3$. Then $\ell=r=k-3$. By the choice of $C$, we may assume that if there exists an edge contained in $V(C)$, then there is exactly one such edge and it belongs to $E(C)$. Thus \[E'(H)=\{e \in E(H)-E(C): e \cap (V(H) - V(C)) \neq \emptyset\} = E(H) - E(C).\]
For simplicity, we write $d_{E'}(v)= d_{E'(H)}(v)$. 
We show the following.
\begin{equation}\label{fnotinC}\hbox{If $u \in V(H) - V(C)$ and $d_{E'}(u) \geq 2$, then for every edge $f \in E'(H)$ containing $u$, $f \cap V(C) = \emptyset$.}\end{equation}

 Suppose this is not true, say $u$ belongs to edges $f_1, f_2 \in E'(H)$, and $f_1 \cap V(C) \neq \emptyset$. Without loss of generality, say $|f_1 \cap V(C)| \geq |f_2 \cap V(C)|$, and say $v_\ell \in f_1$. If it exists, let $u' \in f_2 - V(C) - \{u\}$. One of $e_{\ell}$ and $e_1$ is not identically the set $V(C)$, and hence contains a vertex $w \notin V(C)$. By symmetry, assume it is $e_\ell$.  Then $w, e_\ell, v_1, e_1, \ldots, v_\ell, f_1, u, f_2, u'$ is either a Berge path with $r+3 = k$ vertices or $u' = w$ and it is a longer Berge cycle than $C$.  Otherwise, if $f_2 -V(C) = \{u\}$, then both $f_1$ and $f_2$ contain $r-1$ vertices in $C$. For  distinct $v_i \in f_1 \cap V(C), v_j \in f_2 \cap V(C)$, we must have $|i-j| \geq 2$, otherwise we could replace the short segment between $v_i$ and $v_j$ in $C$ with $v_i, f_1, u, f_2, v_{j}$ and obtain a longer Berge cycle. Applying Lemma \ref{lem:intC} with $q=r-1$ and $s=1$, $X = f_1 \cap V(C), Y = f_2 \cap V(C)$, we have $r=\ell \geq 2(r-1)$, a contradiction. This proves~\eqref{fnotinC}. Now we will show that edges in $E'(H)$ share at most one vertex in $V(H) - V(C)$.

\begin{equation}\label{eq:E'deg2}
\mbox{\shortstack{Every edge $e \in E'(H)$ contains at most one vertex $v \in V(H)-V(C)$ with $d_{E'}(v) \geq 2$.}}    
\end{equation}

Suppose some edge $f_1 \in E'(H)$ contains at least two vertices, $u, u'$ with degree at least 2 in $E'(H)$. By~\eqref{fnotinC}, $f_1 \cap V(C) = \emptyset$. Let $P$ be a shortest semi path from $V(f_1)$ to $V(C) \cup E(C)$. By the minimality of $P$, $P$ contains at most one of $u, u'$, say it does not contain $u'$. Moreover, $f_1 \notin E(P)$. If $P$ ends with a vertex in $V(C)$, let it be $v_i$. If it ends with an edge in $E(C)$, let it be $e_i$.

Suppose $P$ contains at least $2$ vertices outside of $C$. The Berge path 
\[u', f_1, P, v_i, e_{i-1}, v_{i-1}, \ldots, v_{i+1}\] contains at least $r+3=k$ vertices, a contradiction.

Therefore we may assume either $P = w,e_i$ or $P = w, f_2, v_i$ for some $w \in f_1 - V(C)$, $f_2 \in E'(H) - \{f_1\}$. In the latter case, we have $d_{E'}(w) \geq 2$ but $f_2 \cap V(C) \neq \emptyset$, contradicting~\eqref{fnotinC}. So assume $P = w, e_i$. Let $f_2 \in E'(H)$ contain $u'$. Since $f_2 \cap V(C) = \emptyset$ and $r \geq 3$, there exists a vertex $w' \in f_2 - \{w, u'\}$. We instead get the Berge path \[w', f_2, u', f_1, w, e_i, v_i, e_{i-1}, \ldots, v_{i+1}\] with $k$ vertices. This proves~\eqref{eq:E'deg2}.

Define $\mathcal U = \{u \in V(H) - V(C): d_{E'}(u) \geq 2\}$, and for each $u \in \mathcal U$, define $S(u)=\{e \in E'(H): u \in e\}$. By~\eqref{eq:E'deg2}, each pair of edges in $S(u)$ intersect in $V(H) - V(C)$ at exactly $u$. Moreover, for distinct $u, u' \in \mathcal U$, $S(u)$ and $S(u')$ are both edge-disjoint and vertex-disjoint: each edge in $S(u)$ has a unique vertex of degree at least $2$ in $E'(H)$, and so edges in $S(u)$ cannot also contain $u'$. Moreover, all vertices in $\bigcup_{f \in S(u)} (f -V(C)) - \{u\}$ belong to exactly one $E'(H)$ edge which is in $S(u)$. Thus these vertices cannot also belong to $\bigcup_{f \in S(u')} (f -V(C)) - \{u'\}$.

 If $\mathcal U$ is nonempty, let $E''(H)$ be obtained from $E'(H)$ by removing all but one edge from each set $S(u)$ for every $u \in \mathcal U$. By construction, $d_{E''(H)}(v) \leq 1$ for all $v \in V(H) - V(C)$. Moreover, say we removed all edges in $S(u)$ except for some edge $f_1 \in E''(H)$. Then every vertex in $\bigcup_{f \in S(u), f \neq f_1} (f -V(C)) - \{u\}$ has degree $0$ in $E''(H)$. Let $\mathcal X$ be the set of these vertices. 

Suppose $|E'(H)| - |E''(H)| = t$ for some $t \geq 1$. That is, we removed a total of $t$ edges from the $S(u)$'s to obtain $E''(H)$. We have $|\mathcal X| = t(r-1)$. Let $f \in E''(H)$ be an edge in $S(u)$ for some $u \in \mathcal U$ that did not get deleted. The $r$ vertices in $f$ belong to no other edges of $E''(H)$. We have

\begin{eqnarray*}|E(H)| &=& |E(C)| + t+|E''(H)|\\
&\leq& |E(C)| + t + |\{f\}| +\sum_{v \in V(H)-V(C), v \notin f} d_{E''(H)}(v) \\
&\leq& |E(C)| + t +1+|V(H) - V(C)| - r - |\mathcal X|\\
&\leq& r + t + 1+ (n-r) -r- t(r-1)\\
& \leq & r+1+1+(n-r) - r- (r-1)\\
& \leq & n-r\\
&< & {k-3 \choose r} + n-(k-3),\end{eqnarray*}
a contradiction. Thus $E'(H) = E''(H)$, and so no such $S(u)$ exist. That is, 

\begin{equation}\label{eq:deg1}\hbox{every vertex $u \in V(H) - V(C)$ has $d_{E'}(u) \leq 1$.}\end{equation} 

If every edge $e \in E'(H)$ contains at least $2$ vertices in $V(H) - V(C)$, then 
\[|E(H)| \leq |E(C)| + |E(H')| \leq r + \frac{n-r}{2} \leq {k-3 \choose r} + n-(k-3),\]

a contradiction for $n \geq n_0 \geq 3r-2$. So assume some edge $e' \in E'(H)$ contains $r-1$ vertices in $C$. 

Now we show that 
\begin{equation}\label{eq: separate}
|\bigcup_{i=1}^r e_i-V(C)| \geq r-1.
\end{equation}

Since $\ell=r$, there is at most one edge from $E(C)$ that contains no vertex outside of $C$. Thus if $e_i \cap e_j-V(C)=\emptyset$ for every pair $e_i,e_j \in E(C)$, then \eqref{eq: separate} follows immediately. Therefore, we assume that $e_i \cap e_j-V(C) \ne \emptyset$ for some $e_i,e_j \in E(C)$. Let $u \in e_i \cap e_j-V(C)$. We show

First suppose $|j-i|>1$. Since $e'$ contains all but one vertex from $V(C)$, we have either $v_i,v_j \in e'$ or $v_{i+1}, v_{j+1} \in e'$. Without loss of generality, assume $v_i,v_j \in e'$. Then \[v_i, e', v_j, e_{j-1},\ldots, v_{i+1}, e_i, u, e_j, v_{j+1}, e_{j+1}, \ldots, v_i\]
is a Berge cycle longer than $C$, a contradiction.

Hence $|i-j| = 1$. By symmetry, we may assume $j=i+1$. If $v_i, v_{i+1} \in e'$, then replacing the segment $v_i, e_i, v_{i+1}, e_{i+1}, v_{i+2}$ on $C$ by $v_i, e', v_{i+1}, e_i, u, e_{i+1}, v_{i+2}$, we obtain a Berge cycle longer than $C$, a contradiction. Similar for $v_{i+1}, v_{i+2} \in e'$. Thus, we must have $v_{i+1} \notin e'$. And moreover, since $e'$ only omits one vertex from $V(C)$, $e_i, e_{i+1}$ are the only pair of edges in $E(C)$ that may intersect outside of $C$. 

If there exists an edge in $ E(C)$ that contains no vertex outside of $C$, let $e_h=\{v_1, v_2, \ldots, v_r\}$ be that edge. We have $h \neq i, i+1$. Since $v_{i+1}$ is the only vertex from $V(C)$ that is not contained in $e'$, we have $v_h, v_{h+1} \in e'$. Then replacing the segment $v_h, e_h, v_{h+1}$ on $C$ by $v_h, e', v_{h+1}$, and the segment $v_i, e_i, v_{i+1}, e_{i+1}, v_{i+2}$ on $C$ by $v_i, e_h, v_{i+1}, e_i, u, e_{i+1}, v_{i+2}$, we obtain a Berge cycle longer than $C$, a contradiction. Thus, every edge from $E(C)$ contains a vertex outside of $C$. We have shown $e_{i'}\cap e_{j'}-V(C) =\emptyset$ for any pair $e_{i'}, e_{j'} \in E(C)-\{e_i, e_{i+1}\}$. We obtain $|\bigcup_{i=1}^r e_i-V(C)| \geq r-2+1=r-1$, as desired. This proves \eqref{eq: separate}.


Now we show that 
\begin{equation}\label{eq: d=0}
d_{E'}(u)=0 \ \mbox{for any} \ u \in \bigcup_{i=1}^r e_i-V(C).
\end{equation}
Suppose \eqref{eq: d=0} is not true, and some $u$ is contained in $e_i$ and $f$ for some $f \in E'(H)$. Among all such $f$, choose $f$ to maximize $|f \cap V(C)|$. If $|f \cap V(C)| = r-1$, then either $v_i$ or $v_{i+1}$ belong to $f$. By symmetry assume $v_i \in f$. By replacing the segment $v_i, e_i, v_{i+1}$ on $C$ with $v_i, f, u, e_i, v_{i+1}$, we obtain a Berge cycle longer than $C$, a contradiction.
So assume $|f \cap V(C) | \leq r-2$ and let $w \in f-V(C)-\{u\}$. Recall that $e' \in E'(H)$ contains $r-1$ vertices in $V(C)$. By the choice of $f$, $u \notin e'$. Let $z \in e' - V(C)$. Since $e'$ miss only one vertex from $V(C)$, either $v_i \in e'$ or $v_{i+1} \in e'$. Without loss of generality, assume $v_i \in e'$.

 Then \[ z, e', v_i, e_{i-1}, \ldots, e_{i+1}, v_{i+1}, e_i, u, f, w\] is either a Berge path with $\ell+3=k$ vertices or a Berge cycle with $\ell+1$ vertices (if $z = w)$, a contradiction. This proves~\eqref{eq: d=0}. We have shown that every edge in $E'(H)$ contains a vertex $v \notin V(C)$, and moreover this vertex does not belong to any $e_i$ and has $d_{E'}(v) = 1$. It follows that
\begin{align*}
    |E(H)| & = |E(C)|+|E'(H)|\\
    & \leq r+\sum_{ v \in V(H)- \bigcup_{i=1}^r e_i} d_{E'}(v)\\
    & \leq r+n-|V(C)|-|\bigcup_{i=1}^r e_i-V(C)|\\
    & \leq r+n-r-(r-1)=n-r+1 \leq {k-3 \choose r} + n-(k-3),
\end{align*}
a contradiction. This completes the proof.

\end{proof}
\end{document}